\documentclass[a4paper,10pt]{amsart}
\usepackage{etex}

\usepackage[T1]{fontenc}
\usepackage{lmodern}
\usepackage[utf8]{inputenc}
\usepackage{amssymb}
\usepackage{enumitem}
\usepackage{mathrsfs,dsfont,mathtools}
\usepackage{nicefrac}

\usepackage{xcolor}

\usepackage{tikz}
\usetikzlibrary{matrix,arrows,calc}
\tikzset{cd/.style=matrix of math nodes,row sep=2em,column sep=2em, text height=1.5ex, text depth=0.5ex}
\tikzset{cdar/.style=->,auto}
\tikzset{mid/.style={anchor=mid}} 
\tikzset{dar/.style={double,double equal sign distance,-implies}}
\tikzset{narrowfill/.style={inner sep=1pt, fill=white}}

\usepackage[pdftitle={Cuntz Splice invariance for purely infinite graph algebras},
  pdfauthor={Rasmus Bentmann},
  pdfsubject={Mathematics}
]{hyperref}
\usepackage[lite]{amsrefs}

\newcommand*{\arxiv}[1]{\href{http://www.arxiv.org/abs/#1}{arXiv: #1}}

\renewcommand{\PrintDOI}[1]{\href{http://dx.doi.org/\detokenize{#1}}{doi: \detokenize{#1}}%
  \IfEmptyBibField{pages}{, (to appear in print)}{}}

\usepackage{microtype}

\usepackage[all]{xy}

\numberwithin{equation}{section}
\theoremstyle{plain}
\newtheorem{theorem}[equation]{Theorem}
\newtheorem{lemma}[equation]{Lemma}
\newtheorem{proposition}[equation]{Proposition}
\newtheorem{corollary}[equation]{Corollary}
\theoremstyle{definition}
\newtheorem{definition}[equation]{Definition}

\theoremstyle{remark}

\newcommand*{\Z}{\mathbb Z}

\newcommand*{\T}{\mathbb T}
\newcommand*{\Comp}{\mathbb K}

\newcommand*{\XK}{\textup{XK}}
\newcommand*{\K}{\textup{K}}

\newcommand*{\trans}{\textup{t}}

\newcommand*{\Cst}{\textup C^*}
\newcommand*{\Cstar}{\texorpdfstring{$\textup C^*$\nb-}{C*-}}

\newcommand*{\UNIT}{\mathds{1}}

\newcommand*{\nb}{\nobreakdash}  

\mathtoolsset{mathic}

\DeclareMathOperator{\Ext}{Ext}

\DeclareMathOperator{\coker}{coker}
\DeclareMathOperator{\Prim}{Prim}

\newcommand*{\defeq}{\mathrel{\vcentcolon=}}
\newcommand*{\eqdef}{\mathrel{=\vcentcolon}}

\newcommand{\Asf}{\mathsf{A}}
\newcommand{\Psf}{\mathsf{P}}

\newcommand{\AP}{\mathbb{AP}}
\newcommand{\HS}{\mathbb{HS}}
\newcommand{\Ideals}{\mathbb{I}}

\newdir{ >}{{}*!/-5pt/@{>}}

\begin{document}

\title[Cuntz Splice invariance for purely infinite graph algebras]{Cuntz Splice invariance for purely infinite\\ graph algebras}

\author{Rasmus Bentmann}
\email{rasmus.bentmann@mathematik.uni-goettingen.de}

\address{Mathematisches Institut\\
  Georg-August Universit\"at G\"ottingen\\
  Bunsenstra\ss{}e 3--5\\
  37073 G\"ottingen\\
  Germany}

\begin{abstract}
  We show that the Cuntz Splice preserves the stable isomorphism class of a purely infinite graph \Cstar{}algebra with finitely many ideals.
\end{abstract}

\thanks{This work was supported by the German Research Foundation (DFG) through the project grant ``Classification of non-simple purely infinite \Cstar{}algebras''}



\maketitle

\section{Introduction}
\label{sec:intro}

The assignment $E\mapsto\Cst(E)$ associates to a countable directed graph~$E$ a \Cstar{}algebra~$\Cst(E)$ given as the universal \Cstar{}algebra with certain generators and relations encoded by the graph~$E$. This generalizes a construction of Cuntz and Krieger exhibiting close ties to symbolic dynamics: the stabilized Cuntz--Krieger algebra of a subshift of finite type is an invariant of flow equivalence~\cites{Cuntz-Krieger:topological_Markov_chains,Cuntz:topological_Markov_chains_II}.

It is therefore natural to ask when two graphs~$E$ and~$F$ give rise to Morita equivalent \Cstar{}algebras. In particular, it is desirable to find modifications that can be applied to a graph~$E$, such that the \Cstar{}algebra of the resulting graph~$F$ is always stably isomorphic to~$\Cst(E)$. Various such modifications, or \emph{moves}, have been established (see~\cites{Bates-Pask:Flow_equivalence,Crisp-Gow:Contractible,Sorensen:Geometric_class_simple_graph,Eilers-Restorff-Ruiz-Sorensen:Geom_classif}) and, in some cases, it has even been shown that two graphs~$E$ and~$F$ give rise to stably isomorphic \Cstar{}algebras if and \emph{only if} $F$ is the result of a finite number of permitted modifications applied to~$E$ (see~\cites{Sorensen:Geometric_class_simple_graph,Eilers-Restorff-Ruiz-Sorensen:Geom_classif}).

One important example of such a graph modification is the \emph{Cuntz Splice}. This move does not preserve the flow equivalence class of a subshift of finite type because it reverses the sign of the determinant of the matrix $\UNIT-\Asf$, where~$\Asf$ is the adjacency matrix defining the subshift (see~\cite{Bowen-Franks}). However, the Cuntz Splice \emph{does} preserve the stable isomorphism class of the associated Cuntz--Krieger algebra~$\mathcal O_\Asf$ (this follows, for instance, from the classification theorem in~\cite{Restorff:Classification}).

The more general question of Cuntz Splice invariance for the class of all graph \Cstar{}algebras is currently open. A vital step towards the classification of \emph{unital} graph \Cstar{}algebras (that is, \Cstar{}algebras of graphs with finitely many vertices) established in~\cite{Eilers-Restorff-Ruiz-Sorensen:Geom_classif} was to prove Cuntz Splice invariance for this class (see~\cite{Eilers-Restorff-Ruiz-Sorensen:Geom_classif}*{Proposition~5.8}, or \cite{Rordam:Class_of_CK} for an earlier result along such lines).

On the other hand, if a complete classification invariant has already been established by other means, one would hope to be able to determine whether Cuntz Splice invariance holds for the class of graphs under consideration by comparing the values of the invariant. In this article, we achieve this for the class of graphs whose \Cstar{}algebra is purely infinite and has finitely many ideals.

In~\cite{Bentmann-Meyer:More_general}, we have found a complete stable isomorphism invariant, denoted by~$\XK\delta$, for the class of purely infinite graph \Cstar{}algebras with finitely many ideals. Let~$\Cst(E)$ belong to this class and let~$E_C$ be the result of an application of a Cuntz Splice to the graph~$E$. Then $\Cst(E_C)$ also belongs to the classified class, and we will show that $\Cst(E)$ and $\Cst(E_C)$ are stably isomorphic by verifying that $\XK\delta\bigl(\Cst(E)\bigr)\cong\XK\delta\bigl(\Cst(E_C)\bigr)$.

We would like to regard our result as positive evidence towards the hope that Cuntz Splice invariance holds more generally for all (non-simple, non-unital) graph \Cstar{}algebras, although it is not clear to us how the methods from~\cite{Eilers-Restorff-Ruiz-Sorensen:Geom_classif} should be adapted to the non-unital case.

James Gabe recently announced results that lead to a generalization of the classification theorem from~\cite{Bentmann-Meyer:More_general} to cases of certain infinite primitive ideal spaces. A computation as we perform in the present article should then show that Cuntz Splice invariance also holds for purely infinite graph \Cstar{}algebras with infinitely many ideals.

We give a short description of the flavour of our computations. Let $\Asf^E$ denote the adjacency matrix of the graph~$E$. Recall that the $\K$\nb-theory groups of $\Cst(E)$ are computed as the kernel and the cokernel of the matrix $(\Asf^E-\UNIT)^\trans$ regarded as a map on the free Abelian group on the set of vertices of~$E$ (assuming that all vertices of~$E$ are regular). An application of some row and column manipulations preserving the kernel and cokernel of a given matrix then easily shows that $\K_*\bigl(\Cst(E)\bigr)\cong\K_*\bigl(\Cst(E_C)\bigr)$, that is, the Cuntz Splice preserves the $\K$\nb-theory of the graph \Cstar{}algebra. As we shall describe below, the arguments in this article are a refinement of the above computation.

In~\cite{Bentmann:Real_rank_zero_and_int_cancellation}, we defined the homology theory~$\XK$ for \Cstar{}algebras over a finite space~$X$ taking values in $\Z/2$\nb-graded modules over the integral incidence algebra~$\Z X$ of the partially ordered set (associated to)~$X$. Let $E$ be a graph with an identification $\Prim\bigl(\Cst(E)\bigr)\cong X$. Then, just like above, $\XK\bigl(\Cst(E)\bigr)$ is the homology of a length-one chain complex $\Psf_\bullet^E=(P_1^E\to P_0^E)$ of projective $\Z X$\nb-modules canonically associated to~$E$.

Our main result is that there exists an explicit quasi-isomorphism between the complexes~$\Psf_\bullet^E$ and~$\Psf_\bullet^{E_C}$; this is a stronger statement than the two complexes merely having isomorphic homology modules. While the latter statement says that $\XK\bigl(\Cst(E)\bigr)\cong\XK\bigl(\Cst(E_C)\bigr)$, the former implies that $\XK\delta\bigl(\Cst(E)\bigr)\cong\XK\delta\bigl(\Cst(E_C)\bigr)$, so that the classification results from~\cite{Bentmann-Meyer:More_general} become applicable.

The article is organized as follows. After the two preliminary Sections~\ref{sec:graph_algs} and~\ref{sec:the_invariant} discussing some aspects of graphs \Cstar{}algebras and the classification invariant~$\XK\delta$, respectively, we prove the main result under regularity assumptions in Section~\ref{sec:reg} and in full generality in Section~\ref{sec:gen}.

\subsection*{Acknowledgement}
The author would like to thank James Gabe and Adam P.\ W.\ S\o rensen for helpful discussions.

\section{Graph \Cstar{}algebras}
  \label{sec:graph_algs}

In this section, we gather definitions related to graph \Cstar{}algebras and establish a few results needed later on. Throughout this article, a graph will always be a countable directed graph:

\begin{definition}[Graph notions]
A graph~$E$ is a four-tuple $E=(E^0,E^1,r,s)$ where~$E^0$ and~$E^1$ are countable sets, and $r$ and $s$ are maps from~$E^1$ to~$E^0$. The elements of~$E^0$ are called \emph{vertices}, the elements of~$E^1$ are called \emph{edges}, the map~$r$ is called the \emph{range map}, and the map $s$ is called the \emph{source map}. We say that an edge $e\in E^1$ is an edge \emph{from~$s(e)$ to~$r(e)$} and that \emph{$s(e)$ emits~$e$}.

A \emph{path} in~$E$ is a sequence $\mu = e_1 e_2 \cdots e_n$ of edges $e_i\in E^1$ with $n\geq 1$ and $r(e_i)=s(e_{i+1})$ for all $i=1,2,\ldots, n-1$. Extending the range and source maps to paths by setting $s(\mu)=s(e_1)$ and $r(\mu)=r(e_n)$, a \emph{cycle} is a path~$\mu$ such that $s(\mu)=r(\mu)$ and a \emph{return path} is a cycle $\mu = e_1 e_2\cdots e_n$ such that $r(e_i)\neq r(\mu)$ for $i<n$. For $v,w\in E^0$, we write $v\geq w$ if there is a path from~$v$ to~$w$.

A vertex $v\in E^0$ is called \emph{regular} if the set $s^{-1}(v)$ is finite and non-empty, it is called a \emph{sink} if $s^{-1}(v)$ is empty, and it is called an \emph{infinite emitter} if $s^{-1}(v)$ is infinite. A \emph{breaking vertex} is an infinite emitter $v\in E^0$ such that the number of emitted edges, whose range is equal to~$v$ or is the source of a path to~$v$, is finite and non-zero.

The graph~$E$ satisfies \emph{Condition} (K) if no vertex $v\in E^0$ supports precisely one return path.

The \emph{adjacency matrix} $\Asf^E$ of~$E$ is the $E^0\times E^0$\nb-matrix whose entry at $(u,v)$ is the number of egdes from~$u$ to~$v$.
\end{definition}

\begin{definition}[Cuntz Splice]
Let $E=(E^0,E^1,r,s)$ be a graph and let $v\in E^0$ be a vertex that supports at least two (distinct) return paths.
The \emph{Cuntz Splice $E_C^{v}=E_C=(E_C^0,E_C^1,r_C,s_C)$ of~$E$ at~$v$} is defined by 
\begin{align*}
E_C^0 &\defeq E^0\sqcup\{u_1 , u_2 \} \\
E_C^1 &\defeq E^1\sqcup\{f_1 , f_2 , h_1 , h_2 , k_1 , k_2 \},
\end{align*}
where $r_{C}(e)=r(e)$ and $s_{C}(e)=s(e)$ for $e\in E^1$ and
\begin{align*}
s_{C}(f_1)&=v,\quad s_{C}(f_2)=u_1,\quad s_{C}(h_i)=u_1,\quad s_{C}(k_i)=u_2,\\
r_{C}(f_1)&=u_1,\quad r_{C}(f_2)=v,\quad r_{C}(h_i)=u_i,\quad r_{C}(k_i)=u_i.
\end{align*}
In other words, the graph~$E_C^{v}$ is made from the graph~$E$ by adding a segment of the form
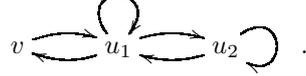
\[
 \xymatrix{v\ar@/^0.5em/[r] & u_1 \ar@/^0.5em/[l] \ar@/^0.5em/[r] \ar@(ul,ur) & u_2 \ar@/^0.5em/[l] \ar@(ur,dr)}\qquad.
\]
\end{definition}

\begin{proposition}
\label{pro:cond_(K)}
 Let $E$ be a graph satisfying Condition~\textup{(K)}. Let $v\in E^0$ support at least two return paths. Then $E_C^v$ satisfies Condition~\textup{(K)}.
\end{proposition}

\begin{proof}
 In the graph $E_C^v$ the distinguished vertex~$v$ as well as the two added vertices~$u_1$ and~$u_2$ support two return paths. Let $w\in (E_C^v)^0\setminus\{v,u_1,u_2\}=E^0\setminus\{v\}$. If $w$ supports two return paths in~$E$ then these are also return paths in~$E_C^v$. If $w$ supports no return path in~$E$ then in particular we cannot have $v\geq w\geq v$, so that $w$ supports no return path in~$E$ either.
\end{proof}

We will use the more common graph \Cstar{}algebra convention, which is \emph{opposite} to the one in Raeburn's monograph~\cite{Raeburn:Book}:

\begin{definition}[Graph \Cstar{}algebra]
Let $E=(E^0,E^1,r,s)$ be a graph. The \emph{graph \Cstar{}algebra} $C^*(E)$ is defined as the universal \Cstar{}algebra generated by mutually orthogonal projections $p_v$, $v \in E^0$, and partial isometries $s_e$, $e \in E^1$, satisfying
\begin{itemize}
	\item $s_e^* s_f=0$ for all $e,f \in E^1$ with $e\neq f$,
	\item $s_e^* s_e=p_{r(e)}$ for all $e\in E^1$,
	\item $s_e s_e^*\leq p_{s(e)}$ for all $e\in E^1$,
	\item $p_v=\sum_{e\in s^{-1}(v)} s_e s_e^*$ for all regular vertices $v\in E^0$.
\end{itemize}
The lattice of closed ideals in $C^*(E)$ will be denoted by~$\Ideals\bigl(\Cst(E)\bigr)$, its primitive ideal space by~$\Prim\bigl(\Cst(E)\bigr)$.
\end{definition}

\begin{definition}[Subsets of graphs]
A subset $K\subseteq E^0$ is \emph{hereditary} if $v\in K$ and $v\geq w$ implies $w\in K$. A subset $K\subseteq E^0$ is \emph{saturated} if the following holds for every regular vertex $v\in E^0$: if $r(e)\in K$ for every $e\in E^1$ with $s(e)=v$ then $v\in K$.  The set of hereditary and saturated subsets of~$E$ is denoted by~$\HS(E)$.

Given a hereditary and saturated subset $H\subseteq E^0$, one defines the set $H_\infty^{\textup{fin}}$ to consist of all infinite emitters in~$E$ that belong to~$E^0\setminus H$ and emit a finite, non-zero number of edges to $E^0\setminus H$. An \emph{admissible pair} for~$E$ is a pair $(H,B)$ consisting of a hereditary and saturated set $H\subseteq E^0$ and an arbitrary subset $B\subseteq H_\infty^{\textup{fin}}$. The set of admissible pairs for~$E$ is denoted by~$\AP(E)$.

A \emph{maximal tail} is a non-empty subset $M\subseteq E^0$ such that the following conditions hold:
\begin{itemize}
 \item if $v\geq w$ and $w\in M$ then $v\in M$;
 \item if $v\in M$ is a regular vertex then there exists $e\in E^1$ with $s(e)=v$ and $r(e)\in M$;
 \item if $v,w\in M$ there exists $y\in M$ such that $v\geq y$ and $w\geq y$.
\end{itemize}
\end{definition}

\begin{theorem}[\cite{Hong-Szymanski:purely_inf_graph_algs}*{Theorem~2.3}]
\label{thm:purely_inf}
 The \Cstar{}algebra $\Cst(E)$ is purely infinite if and only if the following conditions hold:
 \begin{itemize}
  \item the graph~$E$ satisfies Condition \textup{(K)};
  \item there are no breaking vertices in~$E$;
  \item each vertex in each maximal tail~$M$ connects to a cycle in~$M$.
 \end{itemize}

\end{theorem}

\begin{proposition}
  \label{pro:purely_inf}
 Let $E$ be a graph such that $\Cst(E)$ is purely infinite. Let $v\in E^0$ support at least two return paths. Then $\Cst(E_C^v)$ is purely infinite.
\end{proposition}

\begin{proof}
 We must check that the three conditions in Theorem~\ref{thm:purely_inf} are passed on from~$E$ to~$E_C^v$. In Proposition~\ref{pro:cond_(K)}, we have seen that the Cuntz Splice inherits Condition~(K). Clearly no vertex in~$E_C^v$ except~$v$ has a chance of being a breaking vertex. But if~$v$ is regular in~$E$ then it is also regular in~$E_C^v$, and if~$v$ is an infinite emitter in~$E$ then it emits infinitely many edges, both in~$E$ and hence in~$E_C^v$, whose range is equal to~$v$ or supports a path to~$v$. Thus there are no breaking vertices in~$E_C^v$. Finally, let $M$ be a maximal tail in~$E_C^v$. Then $M\cap E^0$ is a maximal tail in~$E$ (in order to check the second condition for the distinguished vertex~$v$, notice that $M\cap E^0$ must completely contain the return paths in~$E$ based at~$v$). This implies that every vertex in~$M\cap E^0$ connects to a cycle in~$M$. If $u_i\in M$ for one (hence both) $i\in\{1,2\}$, then $u_i$ obviously connects to a cycle in~$M$ as well.
\end{proof}

To an admissible pair $(H,B)$ for~$E$, one associates a gauge-invariant ideal $J_{(H,B)}$ in $\Cst(E)$ given as the closed span of a certain set of elements (see \cite{Bates-Hong-Raeburn-Szymanski:Ideal_structure}*{page~6}).

\begin{theorem}[\cite{Bates-Hong-Raeburn-Szymanski:Ideal_structure}*{Theorem 3.6, Corollary 3.10}]
  \label{thm:admissible_pairs}
Let $E$ be a graph satisfying Condition~\textup{(K)}. The assignment $(H,B)\mapsto J_{(H,B)}$ is a bijection from $\AP(E)$ to $\Ideals\bigl(\Cst(E)\bigr)$. One has $J_{(H,B)}\subseteq J_{(H',B')}$ if and only if $H\subseteq H'$ and $B\subseteq H'\cup B'$.
\end{theorem}

We equip the set $\AP(E)$ with the partial ordering~$\leq$ described in the theorem.

\begin{proposition}
  \label{pro:ideals}
Let $E$ be a graph satisfying Condition~\textup{(K)}. Let $v\in E^0$ support at least two return paths. Then
\[
(H,B)\mapsto\begin{cases}(H,B)&\textup{if $v\not\in H$,}\\ (H\cup\{u_1,u_2\},B)&\textup{if $v\in H$}\end{cases}
\]
defines an order isomorphism $\AP(E)\to\AP(E_C^v)$.
\end{proposition}

\begin{proof}
Firstly, the assignment
\[
H\mapsto \bar H\defeq\begin{cases}H&\textup{if $v\not\in H$,}\\ H\cup\{u_1,u_2\}&\textup{if $v\in H$}\end{cases}
\]
is a bijection from the hereditary (and saturated) subsets of~$E$ to the hereditary (and saturated) subsets of~$E_C^v$. Moreover, the set ${\bar H}_\infty^{\textup{fin}}$ (formed in~$E_C^v$) coincides with the set $H_\infty^{\textup{fin}}$ (formed in~$E$). This gives the desired bijection.

Let $(H,B)$ and $(H',B')$ be admissible pairs for~$E$. Then, by definition, $(\bar H,B)\leq (\bar{H'},B)$ if and only if $\bar H\subseteq \bar{H'}$ and $B\subseteq \bar{H'}\cup B'$. The condition $\bar H\subseteq \bar{H'}$ is equivalent to $H\subseteq H'$. Moreover, since $B\cap\{u_1,u_2\}=\emptyset$, the condition $B\subseteq \bar{H'}\cup B'$ is equivalent to $B\subseteq H'\cup B'$. Hence $(\bar H,B)\leq (\bar{H'},B)$ if and only if $(H,B)\leq (H',B)$.
\end{proof}

\begin{corollary}
 \label{cor:ideals}
Let $E$ be a graph satisfying Condition~\textup{(K)}. Let $v\in E^0$ support at least two return paths. Then
\[
J_{(H,B)}\mapsto\begin{cases}J_{(H,B)}&\textup{if $v\not\in H$,}\\ J_{(H\cup\{u_1,u_2\},B)}&\textup{if $v\in H$}\end{cases}
\]
defines an order isomorphism $\Ideals\bigl(\Cst(E)\bigr)\to\Ideals\bigl(\Cst(E_C^v)\bigr)$.
\end{corollary}

\begin{proof}
 Combine Theorem~\ref{thm:admissible_pairs} and Proposition~\ref{pro:ideals}.
\end{proof}

\section{The classification invariant for purely infinite graph \Cstar{}algebras}
  \label{sec:the_invariant}

In this section, we discuss the invariant~$\XK\delta$, which was shown to be a complete (strong) classification invariant for purely infinite graph \Cstar{}algebras with finitely many ideals in~\cite{Bentmann-Meyer:More_general} (using results of Kirchberg~\cite{Kirchberg:Michael}).

For this purpose, we first need to recall the invariant~$\XK$ introduced in~\cite{Bentmann:Real_rank_zero_and_int_cancellation}. Let $X$ be a finite $T_0$\nb-space; it carries a partial ordering called the \emph{spezialization preorder} defined such that $x\geq y$ if and only if $U_x\subseteq U_y$, where $U_x$ denotes the smallest open neighbourhood of the point~$x\in X$. Recall that if~$A$ is a \Cstar{}algebra over~$X$ as defined in~\cite{Meyer-Nest:Bootstrap} (that is, it is equipped with a continuous map $\Prim(A)\to X$) then every open subset~$U$ of~$X$ gives rise to an ideal~$A(U)$ of~$A$. Moreoever, if $U\subseteq V\subseteq X$ are open subsets, we have an ideal inclusion $\iota_U^V\colon A(U)\subseteq A(V)$.

\begin{definition}
  \label{def:XK}
 For a \Cstar{}algebra $A$ over~$X$, the invariant $\XK(A)$ consists of the collection of $\Z/2$-graded Abelian groups $\K_*\bigl(A(U_x)\bigr)$ for $x\in X$ together with the collection of graded group homomorphisms $\K_*\bigl(\iota_{U_x}^{U_y}\bigr)$ for $x\geq y$.
\end{definition}

The assignment~$\XK$ becomes a functor from the category of (separable) \Cstar{}al\-ge\-bras over~$X$ to the category of $X$\nb-diagrams in (countable) $\Z/2$-graded Abelian groups or, equivalently, to the category of (countable) $\Z/2$-graded modules over the integral incidence algebra of the partially ordered set~$X$ (see \cite{Bentmann:Real_rank_zero_and_int_cancellation}*{\S4} for more details). If $M$ is an $X$\nb-diagram as above, we will denote the $\Z/2$-graded Abelian group corresponding to the point~$x\in X$ by~$M_x$ and the homomorphism $M_x\to M_y$ for $x\geq y$ by~$M_{x\to y}$. Thus Definition~\ref{def:XK} says that $\XK(A)_x=\K_*\bigl(A(U_x)\bigr)$ and $\XK(A)_{x\to y}=\K_*\bigl(\iota_{U_x}^{U_y}\bigr)$.

The invariant $\XK\delta(A)$ is only defined for \Cstar{}algebras $A$ over~$X$ such that the projective dimension of the module~$\XK(A)$ is at most~$2$ (see~\cite{Bentmann-Meyer:More_general}*{\S2}). This was shown in~\cite{Bentmann-Meyer:More_general}*{\S5.3} to be the case whenever $A$ is a graph \Cstar{}algebra with primitive ideal space~$X$. Then $\XK\delta(A)$ consists of the invariant $\XK(A)$ equipped with the additional structure of a so-called \emph{obstruction class}~$\delta_A$, which is an element of the $\Z/2$\nb-graded Abelian group 
\begin{equation}
  \label{eq:ext2}
\Ext^2\bigl(\XK(A),\XK(A)[-1]\bigr),
\end{equation}
where $[-1]$ signifies a degree-shift just like in the $\Ext^1$-term in the universal coefficient sequence. Notice that this $\Ext^2$\nb-group is formed in the Abelian target category of~$\XK$ described above. An isomorphism $\XK\delta(A)\cong\XK\delta(A')$ is then simply an isomorphism $\XK(A)\cong\XK(A')$ such that the induced identification
\[
 \Ext^2\bigl(\XK(A),\XK(A)[-1]\bigr)\cong\Ext^2\bigl(\XK(A'),\XK(A')[-1]\bigr)
\]
takes the element~$\delta_A$ to the element~$\delta_{A'}$.

We will omit the general, intrinsic definition of the obstruction class~$\delta_A$ given in \cite{Bentmann-Meyer:More_general}*{Definition~2.16} giving preference to a more explicit description resulting from \cite{Bentmann-Meyer:More_general}*{\S5} in the case that $A=\Cst(E)$ is a graph \Cstar{}algebra. Since the module $\XK_1\bigl(\Cst(E)\bigr)$ is always projective by \cite{Bentmann-Meyer:More_general}*{Lemma~5.18}, one summand of the group \eqref{eq:ext2} vanishes. Elements of the second summand $\Ext^2\bigl(\XK_0(A),\XK_1(A)\bigr)$ can be described as equivalence classes of length-two extensions of $\XK_0(A)$ by $\XK_1(A)$ (see for instance~\cite{MacLane:Homology}*{III.~\S5}). As was shown in \cite{Bentmann-Meyer:More_general}*{Theorem~5.19}, the obstruction class~$\delta_{\Cst(E)}$ is represented by the dual Pimsner--Voiculescu sequence
\[
 0\to\XK_1(A)\to\XK_0(A\rtimes_\gamma\T)
 \xrightarrow{\hat\gamma(1)_*^{-1}-\mathrm{id}}
 \XK_0(A\rtimes_\gamma\T)\to\XK_0(A)\to 0,
\]
where $\gamma$ denotes the canonical gauge action on $A=\Cst(E)$. (As long as one is only interested in an exact sequence, the map in the middle of this sequence is only determined up to a factor of~$\pm 1$. Changing the sign however also replaces the represented $\Ext^2$\nb-class by its additive inverse. Hence there is good reason for our choice of convention in the given setting.)

Recall that \(\Cst(E)\rtimes_\gamma\T\) is the graph \(\Cst\)\nb-algebra of the \emph{skew-product} graph \(E\times_1\Z\) (see \cite{Raeburn-Szymanski:CK_algs_of_inf_graphs_and_matrices}*{\S~3} and \cite{Bates-Hong-Raeburn-Szymanski:Ideal_structure}*{\S~6}). The discussion in \cite{Bentmann-Meyer:More_general}*{\S5} also shows that we can replace $\Cst(E\times_1\Z)$ with the subalgebra $\Cst(E\times_1 N)$ where \(N=\{n\in\Z\mid n\leq 0\}\). This is an improvement because the $\K$\nb-theory of $\Cst(E\times_1 N)$ is more manageable than the one of $\Cst(E\times_1\Z)$. By \cite{Bentmann-Meyer:More_general}*{Theorem~5.3 and (5.17)}, the resulting length-two extension
\[
 \XK_1\bigl(\Cst(E)\bigr)\rightarrowtail\XK_0\bigl(\Cst(E\times_1 N)\bigr)
 \xrightarrow{S-\mathrm{id}}
 \XK_0\bigl(\Cst(E\times_1 N)\bigr)
 \twoheadrightarrow\XK_0\bigl(\Cst(E)\bigr)
\]
still represents the class $\delta_{\Cst(E)}\in\Ext^2\Bigl(\XK_0\bigl(\Cst(E)\bigr),\XK_1\bigl(\Cst(E)\bigr)\Bigr)$. Here $S$ denotes the map induced by the shift map $(e,n)\mapsto (e,n-1)$ on $E\times_1 N$.

\begin{definition}
We define the length-one chain complex
\[
 \Psf_\bullet^E=(P_1^E\xrightarrow{\varphi^E} P_0^E)\defeq\Bigl(\XK_0\bigl(\Cst(E\times_1 N)\bigr)\xrightarrow{S-\mathrm{id}}\XK_0\bigl(\Cst(E\times_1 N)\bigr)\Bigr)
\]
\end{definition}

We observe that the chain complex~$\Psf_\bullet^E$ carries all information about the invariant $\XK\delta\bigl(\Cst(E)\bigr)$: we can recover $\XK_1\bigl(\Cst(E)\bigr)$ and $\XK_0\bigl(\Cst(E)\bigr)$ as the first and zeroth homology modules of this complex; the obstruction class is then represented by the sequence
\[
 \ker(\varphi)\rightarrowtail P_1^E\xrightarrow{\varphi^E} P_0^E\twoheadrightarrow\coker(\varphi).
\]
As a consequence, we obtain the following criterion for isomorphism on the invariant~$\XK\delta$. Recall that a quasi-isomorphism is a chain map inducing isomorphisms on homology in each degree.

\begin{proposition}
 \label{pro:iso_criterion}
Let $E_1$, $E_2$ be graphs such that $\Prim\bigl(\Cst(E_i)\bigr)\cong X$ for both~$i$. If the chain complexes $\Psf_\bullet^{E_1}$ and~$\Psf_\bullet^{E_2}$ are quasi-isomorphic then $\XK\delta\bigl(\Cst(E_1)\bigr)\cong\XK\delta\bigl(\Cst(E_2)\bigr)$.
\end{proposition}

\begin{proof}
A quasi-isomorphism $\psi_\bullet\colon\Psf_\bullet^{E_1}\to\Psf_\bullet^{E_2}$ gives rise to a commutative diagram
\[
  \begin{tikzpicture}[baseline=(current bounding box.west)]
    \matrix(m)[cd,column sep=4em,text height=2.0ex]{
      \XK_1\bigl(\Cst(E_1)\bigr)&P_1^{E_1}&P_0^{E_1}&\XK_0\bigl(\Cst(E_1)\bigr)\\
      \XK_1\bigl(\Cst(E_2)\bigr)&P_1^{E_2}&P_0^{E_2}&\XK_0\bigl(\Cst(E_2)\bigr)\\
    };
    \begin{scope}[cdar]
      \draw[>->] (m-1-1) -- (m-1-2);
      \draw (m-1-2) -- node {\(\varphi^{E_1}\)} (m-1-3);
      \draw[->>] (m-1-3) -- (m-1-4);
      \draw[>->] (m-2-1) -- (m-2-2);
      \draw (m-2-2) -- node {\(\varphi^{E_2}\)} (m-2-3);
      \draw[->>] (m-2-3) -- (m-2-4);
      \draw[dotted] (m-1-1) -- node[swap] {\(\cong\)} (m-2-1);
      \draw (m-1-2) -- node[swap] {\(\psi_1\)}  (m-2-2);
      \draw (m-1-3) -- node {\(\psi_0\)} (m-2-3);
      \draw[dotted] (m-1-4) -- node {\(\cong\)} (m-2-4);
    \end{scope}
  \end{tikzpicture}
\]
The dotted arrows thus yield an isomorphism $\XK\bigl(\Cst(E_1)\bigr)\cong\XK\bigl(\Cst(E_2)\bigr)$. As we have seen above, the two rows represent the obstruction classes~$\delta_{\Cst(E_1)}$ and~$\delta_{\Cst(E_2)}$, respectively. The commutative diagram  thus shows that the obtained dotted isomorphism identifies the obstruction classes (the two $\Ext$\nb-classes coincide by \cite{MacLane:Homology}*{III.~Proposition~5.2}).
\end{proof}

\section{The regular case}
\label{sec:reg}

In this section, we establish the desired result for the case that every vertex in the graph is regular:

\begin{theorem}[]
\label{thm:reg}
Let $E$ be a graph with no sinks and no infinite emitters. Assume that $\Cst(E)$ is purely infinite and has finitely many ideals. Let $v$ be a vertex of~$E$ supporting at least two return paths. Then $\Cst(E)\otimes\Comp\cong\Cst(E_C^{v})\otimes\Comp$.
\end{theorem}

\begin{proof}
The \Cstar{}algebra $\Cst(E_C^{v})$ is purely infinite by Proposition~\ref{pro:purely_inf}. Corollary~\ref{cor:ideals} provides a homeomorphism $\Prim\bigl(\Cst(E_C^{v})\bigr)\to\Prim\bigl(\Cst(E)\bigr)\eqdef X$. By \cite{Bentmann-Meyer:More_general}*{Theorem~5.19}, it suffices to show that $\XK\delta\bigl(\Cst(E)\bigr)\cong\XK\delta\bigl(\Cst(E_C^{v})\bigr)$, which follows from Proposition~\ref{pro:iso_criterion} once we establish a quasi-isomorphism $\Psf_\bullet^{E_C^{v}}\to\Psf_\bullet^E$.

We describe the chain complexes~$\Psf_\bullet^E$ and~$\Psf_\bullet^{E_C^{v}}$ in more detail. Both graphs~$E$ and~$E_C^{v}$ satisfy Condition (K) by Theorem~\ref{thm:purely_inf}. Since there are no infinite emitters, Theorem~\ref{thm:admissible_pairs} provides order isomorphisms $\HS(E)\to\Ideals\bigl(\Cst(E)\bigr)$ and $\HS(E_C^{v})\to\Ideals\bigl(\Cst(E_C^{v})\bigr)$ given by $H\mapsto J_H\defeq J_{(H,\emptyset)}$, and Corollary~\ref{cor:ideals} shows that
\[
J_{H}\mapsto\begin{cases}J_{H}&\textup{if $v\not\in H$,}\\ J_{H\cup\{u_1,u_2\}}&\textup{if $v\in H$}\end{cases}
\]
is a well-defined order isomorphism $\Ideals\bigl(\Cst(E)\bigr)\to\Ideals\bigl(\Cst(E_C^{v})\bigr)$.

Given a point $x\in X$, we let $H_x\in\HS(E)$ and $H_x'\in\HS(E_C^{v})$ denote the hereditary and saturated subsets such that $\Cst(E)(U_x)=J_{H_x}$ and $\Cst(E_C^{v})(U_x)=J_{H_x'}$. By \cite{Raeburn-Szymanski:CK_algs_of_inf_graphs_and_matrices}*{Theorem~3.2}, for $i\in\{0,1\}$, we have
\[
 (P_i^E)_x = \XK_0\bigl(\Cst(E\times_1 N)\bigr)_x = \K_0\bigl(\Cst(E\times_1 N)(U_x)\bigr) \cong \K_0\bigl(\Cst(H_x\times_1 N)\bigr)\cong\Z^{H_x}
\]
and, similarly, $(P_i^{E_C^{v}})_x \cong\Z^{H_x'}$. If $x\geq y$, the map $(P_i^E)_{x\to y}$ is simply the inclusion $\Z^{H_x}\hookrightarrow\Z^{H_y}$ and analogously for $(P_i^{E_C^{v}})_{x\to y}$. Finally, again by \cite{Raeburn-Szymanski:CK_algs_of_inf_graphs_and_matrices}*{Theorem~3.2}, the components $\varphi^E_x\colon (P_1^E)_x\to (P_0^E)_x$ of the module map~$\varphi^E$ are given, under the identifications above, by
\[
 \Z^{H_x}\xrightarrow{(\Asf^E_x-\UNIT)^\trans}\Z^{H_x},
\]
where $\Asf^E_x$ denotes the restriction of the adjacency matrix~$\Asf^E$ to the index set $H_x\subseteq E^0$ (in both the rows and the columns), and $\UNIT$ denotes the identity matrix of the appropriate size. Again, the analogous formula holds for the map~$\varphi_x^{E_C^{v}}$. Here we think of elements of $\Z^{H_x}$ as column vectors and view the matrix $(\Asf^E_x-\UNIT)^\trans$ as a map via matrix multiplication from the left, et cetera. Hence we have implicitly chosen enumerations of the sets~$H_x$ and~$H_x'$. In the following, we assume, if $v\in H_x$, that these enumerations have been chosen to be of the form $H_x=(v,v_1,v_2,\ldots)$ and $H_x'=(u_2,u_1,v,v_1,v_2,\ldots)$, where $(v_1,v_2,\ldots)$ is an arbitrary enumeration of $H_x\setminus\{v\}$. Thus the matrix $\Asf_x^{E_C^v}$ has the following specific block form:
\[
\scalebox{1.5}{$\Asf_x^{E_C^v}$}=
\begin{pmatrix}
\begin{matrix}1&1\\1&1\end{matrix}&
\begin{matrix}0&0&\cdots\\1&0&\cdots\end{matrix}\\
\begin{matrix}0&1\\0&0\\\vdots&\vdots\end{matrix}&
\scalebox{1.5}{$\Asf_x^E$}
\end{pmatrix}.
\]
 
Having fully described the chain complexes~$\Psf_\bullet^E$ and~$\Psf_\bullet^{E_C^{v}}$, we will now define a chain map $\psi_\bullet\colon\Psf_\bullet^{E_C^{v}}\to\Psf_\bullet^E$ and verify that it is a quasi-isomorphism. For this, we need to specify, for every $x\in X$ and $i\in\{0,1\}$, a group homomorphism $(\psi_i)_x\colon (P_i^{E_C^{v}})_x\to (P_i^{E})_x$ such that all face squares of the cube in Figure~\ref{fig:commuting_cube} commute whenever $x\geq y$.
\begin{figure}[htbp]
\[
\begin{tikzpicture}[
cross line/.style={preaction={draw=white, -,
line width=6pt}}]
\matrix (m) [matrix of math nodes,
row sep=3em, column sep=3em,
text height=2.2ex,
text depth=0.25ex]{
& (P_1^{E_C^{v}})_x & & (P_0^{E_C^{v}})_x \\
(P_1^E)_x
& & (P_0^E)_x
\\
& (P_1^{E_C^{v}})_y
& & (P_0^{E_C^{v}})_y
\\
(P_1^E)_y
& & (P_0^E)_y
\\
};
\path[->]
(m-1-2) edge node[auto] {$\varphi^{E_C^{v}}_x$} (m-1-4)
edge node[above left] {$(\psi_1)_x$} (m-2-1)
edge node[below left=5pt] {$(P_1^{E_C^{v}})_{x\to y}$} (m-3-2)
(m-1-4) edge node[auto] {$(P_0^{E_C^{v}})_{x\to y}$} (m-3-4)
edge node[auto] {$(\psi_0)_x$} (m-2-3)
(m-2-1) edge [cross line] node[above right=2pt] {$\varphi^E_x$} (m-2-3)
edge node[left] {$(P_1^E)_{x\to y}$} (m-4-1)
(m-3-2) edge node[below left=5pt] {$\varphi^{E_C^{v}}_y$} (m-3-4)
edge node[auto] {$(\psi_1)_y$} (m-4-1)
(m-4-1) edge node[below=5pt] {$\varphi^E_y$}  (m-4-3)
(m-3-4) edge node[auto] {$(\psi_0)_y$}  (m-4-3)
(m-2-3) edge [cross line] node[above right=2pt] {$(P_0^E)_{x\to y}$} (m-4-3);
\end{tikzpicture}
\]
\caption{A commuting cube}
\label{fig:commuting_cube}
\end{figure}
We know that the front and back square commute because $\varphi^E$ and $\varphi^{E_C^{v}}$ are module maps.  The left and right square commuting means that $\psi_1$ and $\psi_0$ are module maps, and the top and bottom square commuting says that $\psi_\bullet$ is a chain map.

For $x\in X$ such that $H_x\ni v$, we define the maps~$(\psi_i)_x$ from $(P_i^{E_C^{v}})_x\cong\Z^{H_x'}\cong\Z^{\{u_1,u_2\}}\oplus\Z^{H_x}$ to $(P_i^{E})_x\cong\Z^{H_x}$ by the following block matrices:
\[
(\psi_1)_x=
\begin{pmatrix}
\begin{matrix}
0&0\\0&0\\\vdots&\vdots
\end{matrix}
&
\scalebox{1.5}{$\UNIT$}
\end{pmatrix},
\qquad
(\psi_0)_x=
\begin{pmatrix}
\begin{matrix}
-1&0\\0&0\\\vdots&\vdots
\end{matrix}
&
\scalebox{1.5}{$\UNIT$}
\end{pmatrix}.
\]
If $H_x\not\ni v$, we simply let $(\psi_i)_x$ be the identity map from $(P_i^{E_C^{v}})_x\cong\Z^{H_x'}\cong\Z^{H_x}$ to $(P_i^{E})_x\cong\Z^{H_x}$ for both values of~$i$.

We check the commutativity of the left- and right-hand square. If $H_x\ni v$ then also $H_y\ni v$ and the left-hand square commutes because projecting from $\Z^{\{u_1,u_2\}}\oplus\Z^{H_x}$ to $\Z^{H_x}$ and then including into $\Z^{H_y}$ is the same thing as including $\Z^{\{u_1,u_2\}}\oplus\Z^{H_x}$ into $\Z^{\{u_1,u_2\}}\oplus\Z^{H_y}$ and then projecting onto $\Z^{H_y}$. In the right-hand square, if $H_x\ni v$, the two composite maps $\Z^{\{u_1,u_2\}}\oplus\Z^{H_x}\to\Z^{H_y}$ again restrict to the inclusion $\Z^{H_x}\hookrightarrow\Z^{H_y}$ on the second summand. On the first summand $\Z^{\{u_1,u_2\}}$, both take an element $(a,b)$ to the vector $(-a,0,0,\ldots)$. Now we assume that $H_x\not\ni v$. In the case $H_y\not\ni v$, commutativity of both squares is trivial. Otherwise, the left- and right-hand square commute because composing $(\psi_i)_y$ with the inclusion $\Z^{H_x}\hookrightarrow\Z^{\{u_1,u_2\}}\oplus\Z^{H_y}$ simply gives the inclusion $\Z^{H_x}\hookrightarrow\Z^{H_y}$ for both values of~$i$. This shows that the left- and right-hand square always commute.

Next we consider the top square. Commutativity is clear when $H_x\not\ni v$. Otherwise, it comes down to the identity of the two matrix products
\[
\scalebox{1.5}{$(\Asf_x^E-\UNIT)^\trans$}
\begin{pmatrix}
\begin{matrix}
0&0\\0&0\\\vdots&\vdots
\end{matrix}
&
\scalebox{1.5}{$\UNIT$}
\end{pmatrix},
\;
\begin{pmatrix}
\begin{matrix}
-1&0\\0&0\\\vdots&\vdots
\end{matrix}
&
\scalebox{1.5}{$\UNIT$}
\end{pmatrix}
\begin{pmatrix}
\begin{matrix}0&1\\1&0\end{matrix}&
\begin{matrix}0&0&\cdots&\cdots\\1&0&\cdots&\cdots\end{matrix}\\
\begin{matrix}0&1\\0&0\\\vdots&\vdots\end{matrix}&
\scalebox{1.5}{$(\Asf_x^E-\UNIT)^\trans$}
\end{pmatrix},
\]
both of which are indeed equal to
$
\begin{pmatrix}
\begin{matrix}
0&0\\0&0\\\vdots&\vdots
\end{matrix}
&
\scalebox{1.5}{$(\Asf_x^E-\UNIT)^\trans$}
\end{pmatrix}
$.

So far, we have defined a chain map $\psi_\bullet\colon\Psf_\bullet^{E_C^{v}}\to\Psf_\bullet^E$ and it remains to check that it induces isomorphisms on first and zeroth homology. Consider the diagram
\[
  \begin{tikzpicture}[baseline=(current bounding box.west)]
    \matrix(m)[cd,column sep=4em]{
      \ker(\varphi^{E_C^{v}})&P_1^{E_C^{v}}&P_0^{E_C^{v}}&\coker(\varphi^{E_C^{v}})\\
      \ker(\varphi^{E})&P_1^{E}&P_0^{E}&\coker(\varphi^{E}),\\
    };
    \begin{scope}[cdar]
      \draw[>->] (m-1-1) -- (m-1-2);
      \draw (m-1-2) -- node {\(\varphi^{E_C^{v}}\)} (m-1-3);
      \draw[->>] (m-1-3) -- (m-1-4);
      \draw[>->] (m-2-1) -- (m-2-2);
      \draw (m-2-2) -- node {\(\varphi^{E}\)} (m-2-3);
      \draw[->>] (m-2-3) -- (m-2-4);
      \draw (m-1-1) -- node[swap] {\(\psi_1|\)} (m-2-1);
      \draw (m-1-2) -- node[swap] {\(\psi_1\)}  (m-2-2);
      \draw (m-1-3) -- node {\(\psi_0\)} (m-2-3);
      \draw (m-1-4) -- node {\(\bar\psi_0\)} (m-2-4);
    \end{scope}
  \end{tikzpicture}
\]
where \(\psi_1|\) is the restriction of~\(\psi_1\) and \(\bar\psi_0\) is the map induced by~\(\psi_0\). We wish to show that, for every given point $x\in X$, the components \((\psi_1|)_x\) and \((\bar\psi_0)_x\) are invertible.

Let $m\in\ker(\varphi_x^{E_C^{v}})\subseteq (P_1^{E_C^{v}})_x\cong\Z^{\{u_1,u_2\}}\oplus\Z^{H_x}$ satisfy $(\psi_1)_x(m)=0$. Then $m$ must be of the form $(a,b,0,0,\ldots)^\trans$, and $0=\varphi_x^{E_C^{v}}(m)=(b,a,b,0,0\ldots)^\trans$ shows that $a=b=0$ and thus $m=0$. Hence \((\psi_1|)_x\) is injective.

Given $m=(m_1,m_2,\ldots)^\trans\in\ker(\varphi_x^{E})\subseteq (P_1^{E})_x\cong\Z^{H_x}$, we define
\[
 m'\defeq (-m_1,0,m_1,m_2,\ldots)^\trans\in\Z^{\{u_1,u_2\}}\oplus\Z^{H_x}\cong (P_1^{E_C^{v}})_x.
\]
Then $m'\in\ker(\varphi_x^{E_C^{v}})$ and $(\psi_1)_x(m')=m$. Hence \((\psi_1|)_x\) is surjective.

Let $\bar m\in\coker(\varphi_x^{E_C^{v}})$ satisfy $(\bar\psi_0)_x(\bar m)=0$. Then $\bar m$ is represented by an element $m=(a,b,m_1,m_2,\ldots)^\trans\in P_0^{E_C^{v}}\cong\Z^{\{u_1,u_2\}}\oplus\Z^{H_x}$ such that
\[
(\psi_0)_x(m)=(-a+m_1,m_2,m_3,\ldots)^\trans
\]
belongs to the image of~$\varphi_x^E$, that is, $(\psi_0)_x(m)=(\Asf_x^E-\UNIT)^\trans(n)$ for some $n=(n_1,n_2,\ldots)^\trans\in\Z^{H_x}$. We define the element
\[
n'=(b-n_1,a,n_1,n_2,\ldots)^\trans\in\Z^{\{u_1,u_2\}}\oplus\Z^{H_x}\cong(P_1^{E_C^{v}})_x.
\]
Its image under the map~$\varphi_x^{E_C^{v}}$ is
\[
\begin{pmatrix}
\begin{matrix}0&1\\1&0\end{matrix}&
\begin{matrix}0&0&\cdots&\cdots\\1&0&\cdots&\cdots\end{matrix}\\
\begin{matrix}0&1\\0&0\\\vdots&\vdots\end{matrix}&
\scalebox{1.5}{$(\Asf_x^E-\UNIT)^\trans$}
\end{pmatrix}
\begin{pmatrix}
 b-n_1\\a\\n_1\\n_2\\\vdots
\end{pmatrix}
=
\begin{pmatrix}
 a\\b\\a\\0\\0\\\vdots
\end{pmatrix}
+
\begin{pmatrix}
 0\\0\\-a+m_1\\m_2\\m_3\\\vdots
\end{pmatrix}
=m.
\]
This shows that $\bar m=0$. Hence \((\bar\psi_0)_x\) is injective. Surjectivity of \((\bar\psi_0)_x\) follows immediately from surjectivity of~$(\psi_0)_x$. This completes the proof that~$\psi_\bullet$ is a quasi-isomorphism.
\end{proof}

\section{The general case}
\label{sec:gen}

In this section, we remove the regularity assumptions on the graph~$E$ from the previous section. We will use the so-called Drinen--Tomforde Desingularization procedure introduced in~\cite{Drinen-Tomforde:Arbitrary_graph}. It proceeds by ``adding a tail'' to every sink and every infinite emitter in~$E$ (see \cite{Drinen-Tomforde:Arbitrary_graph}*{Definition~2.1}). It is shown in~\cite{Drinen-Tomforde:Arbitrary_graph}*{Theorem~2.11} that this procedure does not change the stable isomorphism class of the graph \Cstar{}algebra. As we shall see, the desired theorem reduces to Theorem~\ref{thm:reg} because Drinen--Tomforde Desingularization commutes with the Cuntz Splice up to stable isomorphism of the associated graph \Cstar{}algebra.

\begin{lemma}
  \label{lem:Desing}
Let $E$ be a graph. Let $v$ be a vertex of~$E$ supporting at least two return paths. Let $F$ be a Drinen--Tomforde Desingularization of~$E$. Then $v$ also supports two return paths in~$F$. Moreover, the Cuntz Splice $F_C^v$ is stably isomorphic to \textup{(}any Drinen--Tomforde Desingularization of\textup{)}~$E_C^v$.
\end{lemma}

\begin{proof}
The first claim is straightforward. Desingularizing a singular vertex \emph{different from~$v$} commutes with the Cuntz Splice at~$v$ even in the strong sense that composing the two procedures in both possible orders results in exactly the same graph (provided, in the case of an infinite emitter, that the same enumeration of outgoing edges is chosen in~$E$ and in~$E_C^v$---differing choices will give different graphs but result in stably isomorphic graph \Cstar{}algebras by \cite{Drinen-Tomforde:Arbitrary_graph}*{Theorem~2.11}). Hence we may assume that~$v$ is an infinite emitter and that all other vertices in~$E$ are regular. If the set of edges emitted from~$v$ in~$E$ is enumerated as $(e_1,e_2,\ldots)$, we choose the enumeration $(f_1,e_1,e_2,\ldots)$ for the set of edges emitted from~$v$ in~$E_C^v$. We will show that $F_C^v$ is stably isomorphic to the corresponding Drinen--Tomforde Desingularization of~$E_C^v$. This implies the second claim of the lemma.

The relevant segments of the two resulting graphs we need to compare look as follows:
\begin{align}
\label{eq:two_graphs}
 \begin{split}
 \xymatrix{\cdots & v_2\ar[l]\ar[d] & v_1\ar[l]\ar[d] & v\ar[d]\ar[l]\ar@/^0.5em/[r]^{f_1} & u_1 \ar@/^0.5em/[l] \ar@/^0.5em/[r] \ar@(dl,dr) & u_2 \ar@/^0.5em/[l] \ar@(ur,dr)\\ \cdots & r(e_3) & r(e_2) & r(e_1),}\\
 \xymatrix{\cdots & v_2\ar[l]\ar[d] & v_1\ar[l]\ar[d] & v\ar[l]\ar@/^0.5em/[r]^{f_1} & u_1 \ar@/^0.5em/[l] \ar@/^0.5em/[r] \ar@(dl,dr) & u_2 \ar@/^0.5em/[l] \ar@(ur,dr)\\ \cdots & r(e_2) & r(e_1).}
 \end{split}
\end{align}
Up to omitted vertices and edges, the first graph represents the Cuntz Splice~$F_C^v$ of~$F$, while the second graph is the desingularization of~$E_C^v$. Notice that a vertex of the form~$r(e_i)$ may agree with~$v$, although this is not reflected in our drawing.

We observe that the subgraph $T=\{v_1\}$ in the second graph is contractible in the sense of~\cite{Crisp-Gow:Contractible}  (although the criteria in~\cite{Crisp-Gow:Contractible} are complicated, they are trivially satisfied for our choice of~$T$ consisting of only one vertex and no edges), and that the contraction procedure described in \cite{Crisp-Gow:Contractible}*{Theorem~3.1} yields the graph
\[
 \xymatrix{\cdots & v_3\ar[l]\ar[d] & v_2\ar[l]\ar[d] & v\ar[d]\ar[l]\ar@/^0.5em/[r] & u_1 \ar@/^0.5em/[l] \ar@/^0.5em/[r] \ar@(dl,dr) & u_2 \ar@/^0.5em/[l] \ar@(ur,dr)\\ \cdots & r(e_3) & r(e_2) & r(e_1),}
\]
which is isomorphic to the first graph above. Hence the two graphs partially shown in~\eqref{eq:two_graphs} have stably isomorphic graph \Cstar{}algebras by \cite{Crisp-Gow:Contractible}*{Theorem~3.1}. Notice that the assumption of no tails in \cite{Crisp-Gow:Contractible}*{Theorem~3.1} does no harm here: we may simply replace all eventual tails with sinks before applying the Crisp--Gow Contraction and desingularize these sinks afterwards.

We remark that it is also possible to write the particular contraction move above as a combination of simpler moves (an out-splitting followed by two reversed out-delays) whose preserving the stabilized graph \Cstar{}algebra was already established in~\cite{Bates-Pask:Flow_equivalence}; see also \cite{Sorensen:Geometric_class_simple_graph}*{Theorem~5.2}.
\end{proof}

\begin{theorem}[]
\label{thm:gen}
Let $E$ be a graph. Assume that $\Cst(E)$ is purely infinite and has finitely many ideals. Let $v$ be a vertex of~$E$ supporting at least two return paths. Then $\Cst(E)\otimes\Comp\cong\Cst(E_C^{v})\otimes\Comp$.
\end{theorem}

\begin{proof}
Combining \cite{Drinen-Tomforde:Arbitrary_graph}*{Theorem~2.11}, Theorem~\ref{thm:reg} and Lemma~\ref{lem:Desing}, we get
\[
 \Cst(E)\otimes\Comp\cong\Cst(F)\otimes\Comp\cong\Cst(F_C^v)\otimes\Comp\cong\Cst(E_C^v)\otimes\Comp,
\]
where $F$ denotes some Drinen--Tomforde Desingularization of~$E$.
\end{proof}

\end{document}